\documentclass[12pt,reqno]{amsart}
\usepackage{amsmath,amsopn,amssymb,amsthm,multicol}
\usepackage{hyperref}
\usepackage{graphicx}

\DeclareMathOperator{\Ad}{Ad}
\DeclareMathOperator{\ad}{ad}
\DeclareMathOperator{\Aut}{Aut}

\DeclareMathOperator{\tr}{tr}

\DeclareMathOperator{\Ric}{Ric}

\newcommand{\fr}{\mathfrak}
\newcommand{\al}{\alpha}
\newcommand{\be}{\beta}

\newcommand{\bb}{\mathbb}

\DeclareMathOperator{\SO}{SO}

\DeclareMathOperator{\Sp}{Sp}

\newcommand{\thickline}{\noalign{\hrule height 1pt}}

 \newtheorem{lemma} {Lemma} [section]
\newtheorem{theorem}[lemma]{Theorem}

\begin{document}

\title{New homogeneous Einstein metrics on Stiefel manifolds} 
\author{Andreas Arvanitoyeorgos, Yusuke Sakane  and Marina Statha}
\address{University of Patras, Department of Mathematics, GR-26500 Rion, Greece}
%\ead{arvanito@math.upatras.gr}
\email{arvanito@math.upatras.gr}
 \address{Osaka University, Department of Pure and Applied Mathematics, Graduate School of Information Science and Technology, Toyonaka, 
Osaka 560-0043, Japan}
%\ead{sakane@math.sci.osaka-u.ac.jp}
 \email{sakane@math.sci.osaka-u.ac.jp}
\address{University of Patras, Department of Mathematics, GR-26500 Rion, Greece}
%\ead{arvanito@math.upatras.gr}
\email{statha@master.math.upatras.gr} 
\medskip
%\noindent
 \thanks{The second  author  was supported by JSPS KAKENHI Grant Number 25400071.}

   \begin{abstract}
%We obtain new invariant Einstein metrics on the Stiefel manifold $V_4\bb{R} ^n$  of all orthonormal $4$-frames in $\bb{R}^n$.  
We consider invariant Einstein metrics on the Stiefel manifold $V_q\bb{R} ^n$  of all orthonormal $q$-frames in $\bb{R}^n$. 
This manifold is diffeomorphic to the homogeneous space $\SO(n)/\SO(n-q)$   and its
isotropy representation contains equivalent summands.  %This causes difficulty in the description of all $\SO(n)$-invariant metrics. 
 We prove, by assuming additional symmetries, that $V_4\bb{R}^n$ $(n\ge 6)$ admits at least 
four $\SO(n)$-invariant Einstein metrics, two of which are Jensen's metrics and the other two are new metrics. Moreover, we prove that $V_5\bb{R}^7$ admits at least 
six invariant Einstein metrics, two of which are Jensen's metrics and the other four are new metrics. 
  

   \end{abstract}

\maketitle

%\today 

 \section{Introduction}
\markboth{Andreas Arvanitoyeorgos, Yusuke Sakane and Marina Statha}{New homogeneous Einstein metrics on Stiefel manifolds and compact Lie groups}

A Riemannian manifold $(M, g)$ is called Einstein if it has constant Ricci curvature, i.e. $\Ric_{g}=\lambda\cdot g$ for some $\lambda\in\bb{R}$. 
  A detailed exposition on Einstein manifolds can be found in \cite{Be}, \cite{W1} and  \cite{W2}.
  General existence results are difficult to obtain and some methods are described
  in \cite{Bom}, \cite{BWZ} and \cite{WZ}.  
Also, in  \cite{ADN1} the authors introduced a method for proving existence of homogeneous
Einstein metrics by assuming additional symmetries.

 In the present article we are interested in invariant Einstein metrics on homogeneous spaces $G/H$ whose isotropy representation is decomposed into a sum of irreducible but possibly equivalent summands.
  By using the method of \cite{ADN1} we study homogeneous Einstein metrics on Stiefel manifolds. % and on compact Lie groups which are not naturally reductive.
  
   The Stiefel manifolds are the homogeneous spaces $V_k\mathbb{R} ^n=\SO(n)/\SO(n-k)$ and the simplest case is
   the sphere $S^{n-1}=\SO(n)/\SO(n-1)$, which is an irreducible symmetric space, therefore it admits up to scale a unique invariant Einstein metric.
Concerning Einstein metrics on other Stiefel manifolds we review the following:
   In \cite{K} S. Kobayashi proved existence of an invariant Einstein metric on the unit tangent bundle $T_1S^n = \SO(n)/\SO(n-2)$.  In \cite{S} A. Sagle proved that the Stiefel manifolds 
   $V_k\mathbb{R} ^n=\SO(n)/\SO(n-k)$ admit at least one homogeneous Einstein metric.  
   For $k\ge 3$ G. Jensen in \cite{J2} found a second metric.
  If $n=3$ the Lie group $\SO(3)$ admits a unique Einstein metric.
  For $n\ge 5$  A. Back and W.Y. Hsiang in \cite{BH} proved that $\SO(n)/\SO(n-2)$ admits
  exactly one homogeneous Einstein metric. The same result was obtained by M. Kerr in \cite{Ke} by proving that the diagonal metrics are the only invariant metrics on $V_2\mathbb{R}^n$ (see also \cite{A1}, \cite{AK}).
  The Stiefel manifold $\SO(4)/\SO(2)$ admits exactly two invariant Einstein metrics (\cite{ADF}).
  One is Jensen's metric and the other one is the product metric on $S^3\times S^2$.
  
  Finally, in \cite{ADN1} the first author, V.V. Dzhepko and Yu. G. Nikonorov proved that
  for $s>1$ and $\ell >k\ge 3$ the Stiefel manifold $\SO(s k+\ell)/\SO(\ell)$ admits at least four
  $\SO(s k+\ell)%\times (\SO(k))^s
  $-invariant Einstein metrics, two of which are Jensen's metrics.
  The special case $\SO( 2 k + \ell)/\SO(\ell)$ admitting at least four
  $\SO(2 k+\ell) $-invariant Einstein metrics was treated in \cite{ADN2}.
  Corresponding results for the quaternionic Stiefel manifolds $\Sp(sk+\ell)/\Sp(\ell)$ were obtained in
  \cite{ADN3}.
  In the present paper we prove that  $\SO(n)/\SO(n-4)$ admits two more $\SO(n)$-invariant Einstein metrics and that
  $\SO(7)/\SO(2)$ admits four more $\SO(7)$-invariant Einstein metrics.

\section{$G$-invariant metrics with additional symmetries}

We review a construction originally introduced in \cite{ADN1}.  Let $G$ be a compact Lie group and $H$ a closed subgroup so that $G$ acts almost effectively on $G/H$.  Let $\fr{g}, \fr{h}$ be the Lie algebra of $G$ and $H$ and let $\fr{g} = \fr{h} \oplus \fr{m}$ be a reductive decomposition of $\fr{g}$ with respect to some $\Ad(G)$-invariant inner product on $\fr{g}$.  The orthogonal complement $\fr{m}$ can be identified with the tangent space $T_{o}(G/H)$, where $o = eH$.  Any $G$-invariant metric $g$ of $G/H$ corresponds to an $\Ad(H)$-invariant inner product $\left\langle \cdot, \cdot \right\rangle$ on $\fr{m}$ and vice versa.  For $G$ semisimple the negative of the Killing form $B$ of $\fr{g}$ is an $\Ad(G)$-invariant inner product on $\fr{g}$, therefore we can choose the above decomposition with the respect to $-B$.  

The normalizer $N_{G}(H)$ of $H$ in $G$ acts on $G/H$ by $(\al, gH) \mapsto g\al^{-1}H$. For a fixed 
$\al\in N_{G}(H)$ this action induces a $G$-equivariant diffeomorphism $\varphi_{\al} : G/H \to G/H$.  Let $K$ be a closed subgroup of $G$ with $H\subset K \subset G$ such that $K = L'\times H'$, where $L', H'$ are subgroups of $G$ with $H = \{e_{L'}\}\times H'$.  It is clear that $K\subset N_{G}(H)$.  Consider the group $L = L'\times \{e_{H'}\}$.  Then the group $\widetilde{G} = G\times L$ acts on $G/H$ by $(\al, \be)\cdot gH = \al g \be^{-1}H$ and the isotropy $\widetilde{H}$ at $eH$ is given as follows:  

\begin{lemma}(\cite[Lemma 2.1]{ADN1})  
The isotropy subgroup $\widetilde{H}$ is isomorphic to $K$. 
\end{lemma} 

Hence, we have the diffeomorphisms 
\begin{equation}\label{diff1}
G/H\cong (G\times L)/K = (G\times L)/ (L'\times H')\cong \widetilde{G}/\widetilde{H}.
\end{equation}

The elements of the set $\mathcal{M}^{G}$, of $G$-invariant metrics on $G/H$, are in $1-1$ correspondence with $\Ad(H)$-invariant inner products on $\fr{m}$.  We now consider  $\Ad(K)$-invariant inner products on $\fr{m}$ (and not only $\Ad(H)$-invariant) with corresponding subset $\mathcal{M}^{G, K}$ of $\mathcal{M}^{G}$.  Let $g\in \mathcal{M}^{G, K}$ and $\al\in K$.  The diffeomorphism $\varphi_{\al}$ is an isometry of $(G/H, g)$.  The action of $\widetilde{G}$ on $(G/H, g)$ is isometric, so any metric in $\mathcal{M}^{G, K}$ can be identified as a metric in $\mathcal{M}^{\widetilde{G}}$, the set of $\widetilde{G}$-invariant metrics on $\widetilde{G}/\widetilde{H}$, and vice versa.  Therefore, we have $\mathcal{M}^{\widetilde{G}}\subset \mathcal{M}^{G}$.   

\smallskip 

We consider the isotropy representation $\chi : H\to\Aut (\fr{m})$ of $G/H$, which is the restriction
of the adjoint representation of $K$ on $\fr{m}$.
We assume that $\chi$ decomposes into a direct sum of irreducible subrepresentations 

$$
\chi \cong \chi_{1}\oplus\cdots\oplus\chi_{s},
$$
giving rise to a decomposition
$$
\fr{m}=\fr{m}_{1}\oplus\cdots\oplus\fr{m}_{s}
$$
into irreducible $\Ad(H)$-modules.
Note that some of the $\chi_{i}$'s can be equivalent so the above decomposition is not unique.  

We now consider the isotropy representation $\psi :\widetilde{H} \to\Aut(\fr{m})$ of
$\widetilde{G}/\widetilde{H}$ and assume corresponding decompositions

\begin{equation}\label{isotr}
\psi \cong \psi _1\oplus\cdots\oplus\psi _k
\end{equation}
and
\begin{equation}\label{submod}
\fr{m}=\widetilde{\fr{m}}_{1}\oplus\cdots\oplus\widetilde{\fr{m}}_{k},\ \ (k\le s)
\end{equation}
into irreducible $\Ad(K)$-modules.

It is possible that for certain examples of homogeneous spaces $G/H$  (e.g. $\SO(n)/\SO(\ell)$) to choose $L', H'$ (hence the subgroup $K$) such that the submodules
$\widetilde{\fr{m}}_{i}$ are not equivalent, therefore we can obtain a unique expression
of $\widetilde{G}$-invariant metrics determined by $\Ad (K)$-invariant scalar products

\begin{equation}\label{diag}
\langle \cdot , \cdot \rangle = y_1(-B)|_{\widetilde{\fr{m}}_{1}} + \cdots +
y_k(-B)|_{\widetilde{\fr{m}}_{k}}.
\end{equation}
Hence one can hope to find Einstein metrics among such {\it diagonal} metrics $g\in\mathcal{M}^{\widetilde{G}}$.

\section{The Ricci tensor for reductive homogeneous spaces}
 We recall an expression for the Ricci tensor for a $G$-invariant Riemannian
metric on a reductive homogeneous space whose isotropy representation
is decomposed into a sum of non equivalent irreducible summands.

Let $G$ be a compact semisimple Lie group, $K$ a connected closed subgroup of $G$  and  
let  $\frak g$ and $\fr{k}$  be  the corresponding Lie algebras. 
The Killing form $B$ of $\frak g$ is negative definite, so we can define an $\mbox{Ad}(G)$-invariant inner product $-B$ on 
  $\frak g$. 
Let $\frak g$ = $\frak k \oplus
\frak m$ be a reductive decomposition of $\frak g$ with respect to $-B$ so that $\left[\,\frak k,\, \frak m\,\right] \subset \frak m$ and
$\frak m\cong T_o(G/K)$.
 We assume that $ {\frak m} $ admits a decomposition into mutually non equivalent irreducible $\mbox{Ad}(K)$-modules as follows: \ 
\begin{equation}\label{iso}
{\frak m} = {\frak m}_1 \oplus \cdots \oplus {\frak m}_q.
\end{equation} 
Then any $G$-invariant metric on $G/K$ can be expressed as  
\begin{eqnarray}
 \langle  \,\,\, , \,\,\, \rangle  =  
x_1   (-B)|_{\mbox{\footnotesize$ \frak m$}_1} + \cdots + 
 x_q   (-B)|_{\mbox{\footnotesize$ \frak m$}_q},  \label{eq2}
\end{eqnarray}
for positive real numbers $(x_1, \dots, x_q)\in\bb{R}^{q}_{+}$.  Note that  $G$-invariant symmetric covariant 2-tensors on $G/K$ are 
of the same form as the Riemannian metrics (although they  are not necessarilly  positive definite).  
 In particular, the Ricci tensor $r$ of a $G$-invariant Riemannian metric on $G/K$ is of the same form as (\ref{eq2}), that is 
 \[
 r=z_1 (-B)|_{\mbox{\footnotesize$ \frak m$}_1}  + \cdots + z_{q} (-B)|_{\mbox{\footnotesize$ \frak m$}_q} ,
 \]
 for some real numbers $z_1, \ldots, z_q$.

Let $\lbrace e_{\alpha} \rbrace$ be a $(-B)$-orthonormal basis 
adapted to the decomposition of $\frak m$,    i.e. 
$e_{\alpha} \in {\frak m}_i$ for some $i$, and
$\alpha < \beta$ if $i<j$. 
We put ${A^\gamma_{\alpha
\beta}}= -B \left(\left[e_{\alpha},e_{\beta}\right],e_{\gamma}\right)$ so that
$\left[e_{\alpha},e_{\beta}\right]%_{\frak n}
= \displaystyle{\sum_{\gamma}
A^\gamma_{\alpha \beta} e_{\gamma}}$ and set 
$\displaystyle{k \brack {ij}}=\sum (A^\gamma_{\alpha \beta})^2$, where the sum is
taken over all indices $\alpha, \beta, \gamma$ with $e_\alpha \in
{\frak m}_i,\ e_\beta \in {\frak m}_j,\ e_\gamma \in {\frak m}_k$ (cf. \cite{WZ}).  
Then the positive numbers $\displaystyle{k \brack {ij}}$ are independent of the 
$-B$-orthonormal bases chosen for ${\frak m}_i, {\frak m}_j, {\frak m}_k$,
and 
%\begin{equation}
$\displaystyle{k \brack {ij}}\ =\ \displaystyle{k \brack {ji}}\ =\ \displaystyle{j \brack {ki}}.  
 \label{eq3}
$

Let $ d_k= \dim{\frak m}_{k}$. Then we have the following:

\begin{lemma}\label{ric2}\textnormal{(\cite{PS})}
The components ${ r}_{1}, \dots, {r}_{q}$ 
of the Ricci tensor ${r}$ of the metric $ \langle  \,\,\, , \,\,\, \rangle $ of the
form {\em (\ref{eq2})} on $G/K$ are given by 
\begin{equation}
{r}_k = \frac{1}{2x_k}+\frac{1}{4d_k}\sum_{j,i}
\frac{x_k}{x_j x_i} {k \brack {ji}}
-\frac{1}{2d_k}\sum_{j,i}\frac{x_j}{x_k x_i} {j \brack {ki}}
 \quad (k= 1,\ \dots, q),    \label{eq51}
\end{equation}
where the sum is taken over $i, j =1,\dots, q$.
\end{lemma} 
Since by assumption the submodules $\fr{m}_{i}, \fr{m}_{j}$ in the decomposition (\ref{iso}) are matually non equivalent for any $i\neq j$, it is $r(\fr{m}_{i}, \fr{m}_{j})=0$ whenever $i\neq j$.  Thus by Lemma \ref{ric2} it follows that    $G$-invariant Einstein metrics on $M=G/K$ are exactly the positive real solutions $g=(x_1, \ldots, x_q)\in\bb{R}^{q}_{+}$  of the  polynomial system $\{r_1=\lambda, \ r_2=\lambda, \ \ldots, \ r_{q}=\lambda\}$, where $\lambda\in \bb{R}_{+}$ is the Einstein constant.

%%% Stiefel manifolds

\section{The Stiefel manifolds $V_k\mathbb{R}^n=\SO(n)/\SO(n-k)$}
    
We first review the isotropy representation of $V_k\mathbb{R}^n=G/H=\SO(n)/\SO(n-k)$.
Let $\lambda _n$ denote the standard representation of $\SO(n)$ (given by the natural action of
$\SO(n)$ on $\mathbb{R}^n$.  If $\wedge ^2\lambda _n$ denotes the second exterior power of $\lambda _n$, then $\Ad ^{\SO(n)}=\wedge ^2\lambda _n$.
The isotropy representation $\chi$ of $\SO(n)/\SO(n-k)$ is characterized by the property
$\left.\Ad ^{\SO(n)}\right |_{\SO(n-k)}=\Ad ^{\SO(n-k)}\oplus\chi$. 
%(cf.\cite{A2}). 
We compute

\begin{equation}\label{isotropy}
\left.\Ad ^{\SO(n)}\right |_{\SO(n-k)}=\wedge ^2\lambda _n\big| _{\SO(n-k)}=\wedge ^2 (\lambda _{n-k}\oplus k)=\wedge ^2\lambda _{n-k}\oplus \wedge ^2k\oplus (\lambda _{n-k}\otimes k),
\end{equation}
    where $k$ denotes the trivial $k$-dimensional representation.
    Therefore the isotropy representation is given by
    $$
    \chi = \wedge ^2k\oplus \underbrace{\lambda _{n-k}\oplus \cdots\oplus\lambda _{n-k}}_{k}.
$$
We can further decompose $\wedge ^2k$ into a sum of ${k}\choose{2}$ $1$-dimensional sub-representations, so we obtain that
\begin{equation}\label{Stiefel}
\chi = \underbrace{1\oplus \cdots\oplus 1}_{{k}\choose{2}}\oplus \underbrace{\lambda _{n-k}\oplus \cdots\oplus\lambda _{n-k}}_{k}.
\end{equation} 
This decomposition induces an $\Ad (H)$-invariant decomposition of $\fr{m}$ given by

\begin{equation}\label{decomp}
\fr{m}=\fr{m}_1\oplus\cdots\oplus\fr{m}_s,
\end{equation}
where the first ${k}\choose{2}$ $\Ad (H)$-submodules are $1$-dimensional and the rest
$k$ $\Ad (H)$-submodules are $(n-k)$-dimensional.
Note that decomposition (\ref{Stiefel}) contains equivalent summands so a complete description of all $G$-invariant metrics associated to decomposition (\ref{decomp}) is rather hard.

\smallskip
We embed the group $\SO(n-k)$ in $\SO(n)$ as 
$\begin{pmatrix}
0_k & 0\\
0 & C
\end{pmatrix}
$
where $C \in \SO(n-k)$.  The Killing form of $\fr{so}(n)$ is $B(X, Y)=(n-2)\tr XY$.
Then with respect to
$-B$ the subspace $\fr{m}$
 may be identified with the set 
 of matrices of the form
$$
\begin{pmatrix}
D_k & A\\
-A^t & 0_{n-k}
\end{pmatrix},
$$
in $\fr{so}(n)$, 
where $D_k \in \fr{so}(k)$ %=\diag (0,\dots ,0)$
 and $A$ is an $k\times (n-k)$ real matrix.
Let $E_{ab}$ denote the $n\times n$ matrix with $1$ at the $(ab)$-entry and $0$ elsewhere.
Then the set
$\mathcal{B}=\{e_{ab}=E_{ab}-E_{ba}: 1\le a\le k,\ 1\le a<b\le n\}$
constitutes a $-B$-orthogonal basis of $\fr{m}$.
Note that $e_{ba}=-e_{ab}$, thus we have the following:

\begin{lemma}\label{brac}
If all four indices are distinct, then the Lie brakets in $\mathcal{B}$ are zero.
Otherwise,
$[e_{ab}, e_{bc}]=e_{ac}$, where $a,b,c$ are distinct.
\end{lemma}

%Let $\mathcal{B}_1=\{e_{ab}: 1\le a<b\le k\}$, $\mathcal{B}_2=\{e_{ab}: 1\le a\le k, k+1\le b\le n\}$ and $\mathcal{B}_3=\{e_{ab}: k+1\le a<b\le n\}$. Then $\mathcal{B}=\mathcal{B}_1\cup\mathcal{B}_2$ and $\mathcal{B}_1, \mathcal{B}_3$ are $Q$-orthonormal bases of $\mathfrak{o}(k)$ and $\mathfrak{so}(n-k)$ respectively. Furthermore, the set $\mathcal{B}_2$ is a $Q$-orthonormal basis of the tangent space to the Grassmannian ${\rm Gr}_k\mathbb{R} ^n = \SO(n)/\s(\OOO(k)\times \OOO(n-k))$ at the identity coset.

%\begin{lemma}\label{relations}
%The sets $\mathcal{B}_1, \mathcal{B}_2, \mathcal{B}_3$ satisfy the following relations: 
%\begin{eqnarray*}  &(a)&\ [\mathcal{B}_1, \mathcal{B}_1]\subset \mathcal{B}_1, \quad (b)\ \ [\mathcal{B}_2, \mathcal{B}_2]\subset \mathcal{B}_1\oplus \mathcal{B}_3, \quad
%(c)\ \ [\mathcal{B}_3, \mathcal{B}_3]\subset \mathcal{B}_3\\&(d)&\ [\mathcal{B}_1,\mathcal{B}_2]\subset \mathcal{B}_2, \quad
%(e)\ \  [\mathcal{B}_1, \mathcal{B}_3]=(0) \quad
%(f) \ \ [\mathcal{B}_2, \mathcal{B}_3]\subset\mathcal{B}_2. \end{eqnarray*}\end{lemma}

%\begin{proof}  Use properties of the Grassmannian.  I will write it.\end{proof}

%\bigskip
%Next, we will  a Ricci tensor form from  Besse

\section{The Stiefel manifolds $V_{k_1+k_2}\mathbb{R}^{k_1+k_2+k_3} = \SO(k_1+k_2+k_3)/\SO(k_3)$}

We first consider the homogeneous space $G/K= \SO(k_1+k_2+k_3)/(\SO(k_1)\times\SO(k_2)\times\SO(k_3))$, where the embedding of $K$ in $G$ is diagonal. 
Let $\sigma _i:\SO(k_1)\times\SO(k_2)\times\SO(k_3)\to \SO(k_i)$  be the projection onto
the factor $SO(k_i)$ ($i=1,2,3$) and let $p_{k_i}=\lambda _{k_i}\circ\sigma$.
Then by a similar computation as in (\ref{isotropy}) we obtain that
$$
\left.\Ad ^G\right |_K=\Ad ^K\oplus (p_{k_1}\otimes p_{k_2})\oplus (p_{k_1}\otimes p_{k_3})
\oplus (p_{k_2}\otimes p_{k_3}),
$$
so the isotropy representation of $G/K$ has the form
$$
\chi\cong \chi _{12}\oplus\chi _{13}\oplus\chi _{23},
$$
and the tangent space $\fr{m} $ of $G/K$ decomposes into  three  $\Ad(K)$-submodules 
 $$\fr{m} = \fr{m}_{12}\oplus  \fr{m}_{13}\oplus  \fr{m}_{23}. 
 $$
% Let $\sigma _i:\SO(k_1)\times\SO(k_2)\times\SO(k_3)\to \SO(k_i)$  be the projection onto the factor $SO(k_i)$ ($i=1,2,3$) and let $p_{k_i}=\lambda _{k_i}\circ\sigma$. Then by a similar computation as in (\ref{isotropy}) we obtain that
%$$ \left.\Ad ^G\right |_K=\Ad ^K\oplus (p_{k_1}\otimes p_{k_2})\oplus (p_{k_1}\otimes p_{k_3}) \oplus (p_{k_2}\otimes p_{k_3}), $$ therefore, the tangent space $\fr{m} $ of $G/K$ decomposes into the three  submodules  $$\fr{m} = \fr{m}_{12}\oplus  \fr{m}_{13}\oplus  \fr{m}_{23}.  $$
% We adapt the basis $\mathcal{B}$ to the above decomposition so that
In fact, 
  $ \fr{m}$ is given by  $\fr{k}^{\perp} $ in $ \fr{g} = \fr{so}(k_1+ k_2+k_3)$ with respect to  $-B$. If we denote by $M(p,q)$ the set of all $p \times q$ matrices, then we see that 
  $ \fr{m}$ is given by 
  \begin{equation*}
\fr{m}=  \left\{\begin{pmatrix}
 0 & {A}_{12} & {A}_{13}\\
 -{}^{t}_{}\!{A}_{12} & 0 & {A}_{23}\\
 -{}^{t}_{}\!{A}_{13} & -{}^{t}_{}\!{A}_{23} & 0 
 \end{pmatrix} \  \Big\vert \  {A}_{12} \in M(k_1, k_2), {A}_{13} \in M(k_1, k_3), {A}_{23} \in M(k_2, k_3) \right\} 
 \end{equation*}
and we have 
   \begin{equation*}
 \fr{m}_{12}= \begin{pmatrix}
 0 & {A}_{12} & 0\\
 -{}^{t}_{}\!{A}_{12} & 0 &0\\
0  & 0 & 0 
 \end{pmatrix},  \quad  
 \fr{m}_{13}= \begin{pmatrix}
 0 & 0 &{A}_{13}\\
0 & 0 &0\\
 -{}^{t}_{}\!{A}_{13}   & 0 & 0 
 \end{pmatrix},  \quad  
 \fr{m}_{23}= \begin{pmatrix}
 0 & 0 & 0\\
0 & 0 &{A}_{23}\\
0  &  -{}^{t}_{}\!{A}_{23} & 0 
 \end{pmatrix}.  
 \end{equation*}
 Note that the action of $\Ad(k)$ ($k \in K$) on $ \fr{m}$ is given by 
 \begin{equation*}
\Ad(k) \begin{pmatrix}
 0 & {A}_{12} & {A}_{13}\\
 -{}^{t}_{}\!{A}_{12} & 0 & {A}_{23}\\
 -{}^{t}_{}\!{A}_{13} & -{}^{t}_{}\!{A}_{23} & 0 
 \end{pmatrix}  = 
  \begin{pmatrix}
 0 &{}^t h_1 {A}_{12} h_2 & {}^t h_1{A}_{13} h_3\\
 -{}^{t}_{}h_2 {}^{t}_{}\!{A}_{12}h_1 & 0 & {}^t h_2{A}_{23} h_3\\
 -{}^{t}_{}h_3{}^{t}_{}\!{A}_{13} h_1 & -{}^{t}_{}h_3{}^{t}_{}\!{A}_{23} h_2& 0 
 \end{pmatrix}, 
  \end{equation*}
 where $ \begin{pmatrix}
 h_1 & 0& 0\\
0 & h_2 &0\\
0  & 0 & h_3 
 \end{pmatrix} \in K$.  Thus  the irreducible submodules  $\fr{m}_{12}$,  $\fr{m}_{13}$ and  $\fr{m}_{23}$  are  mutually non equivalent. 
 
 We now consider the Stiefel manifold $G/H=\SO(k_1+k_2+k_3)/\SO(k_3)$ and we take into account
 the diffeomorphism
 \begin{equation}\label{diff2}
 G/H=(G\times \SO(k_1)\times\SO(k_2))/((\SO(k_1)\times\SO(k_2))\times\SO(k_3))=\widetilde{G}/
 \widetilde{H}
 \end{equation}
 (compare with (\ref{diff1}) where $L=\SO(k_1)\times\SO(k_2)$).
 Let
 \begin{equation}\label{decomp1}
 \fr{p}=\fr{so}(k_1)\oplus\fr{so}(k_2)\oplus  \fr{m}_{12}\oplus  \fr{m}_{13}\oplus  \fr{m}_{23} 
 \end{equation}
 be an $\Ad(\SO(k_1)\times\SO(k_2)\times\SO(k_3))$-invariant decomposition
of the tangent space $\fr{p}$ of $G/H$ at $eH$, where the corresponding
 submodules are non equivalent.
 Then we consider a subset of all $G$-invariant metrics on $G/H$ 
 determined by the $\Ad(\SO(k_1)\times\SO(k_2)\times\SO(k_3))$-invariant scalar products
 on $\fr{p}$ given by
 \begin{equation}\label{metric1}
 \langle \   ,\  \rangle =  x_1 \, (-B) |_{\fr{so}(k_1)}+ x_2 \, (-B) |_{ \fr{so}(k_2)}
 + x_{12} \,  (-B) |_{ \fr{m}_{12}}+ x_{13} \,  (-B) |_{ \fr{m}_{13}} + x_{23} \,  (-B) |_{ \fr{m}_{23}}
 \end{equation}
 for $k_1 \geq 2$, $k_2 \geq 2$ and $k_3 \geq 1$.
 
Let $\fr{h}$ be the Lie algebra of $H$ and we set in the decomposition (\ref{decomp1}) 
$\fr{so}(k_1)=\fr{m}_1,  
\fr{so}(k_2)=\fr{m}_2$. 
Then by using Lemma \ref{brac} we obtain that the following relations hold:
 
% $$
%[ \fr{m}_1, \fr{m}_1] = \fr{m}_1,  \quad   [ \fr{m}_2, \fr{m}_2] = \fr{m}_2,  \quad [ \fr{m}_1, \fr{m}_{12}] = \fr{m}_{12},$$
% $$ [ \fr{m}_1, \fr{m}_{13}] = \fr{m}_{13},  \quad 
% [\fr{m}_2, \fr{m}_{12}] =  \fr{m}_{12}, \quad [ \fr{m}_2, \fr{m}_{23}] = \fr{m}_{23},$$ 
%   $$ [ \fr{m}_{12}, \fr{m}_{12}] \subset \fr{m}_{1}\oplus  \fr{m}_{2},   \quad   [ \fr{m}_{13}, \fr{m}_{13}] \subset \fr{m}_{1} \oplus \fr{h},  \quad 
% [\fr{m}_{23}, \fr{m}_{23}] \subset \fr{m}_{2} \oplus \fr{h}, $$ 
 %  $$ [ \fr{m}_{12}, \fr{m}_{23}] \subset \fr{m}_{13},   \quad   [ \fr{m}_{13}, \fr{m}_{23}] \subset \fr{m}_{12},  \quad 
% [\fr{m}_{12}, \fr{m}_{13}] \subset \fr{m}_{23}.$$ 
 
\begin{center}
\begin{tabular}{ l l l }
$[ \fr{m}_1, \fr{m}_1] = \fr{m}_1,$ & $[ \fr{m}_2, \fr{m}_2] = \fr{m}_2,$ & $[ \fr{m}_1, \fr{m}_{12}] = \fr{m}_{12},$\\ $[ \fr{m}_1, \fr{m}_{13}] = \fr{m}_{13},$ & $[\fr{m}_2, \fr{m}_{12}] =  \fr{m}_{12},$ & $[ \fr{m}_2, \fr{m}_{23}] = \fr{m}_{23},$\\
$[ \fr{m}_{12}, \fr{m}_{12}] \subset \fr{m}_{1}\oplus  \fr{m}_{2},$ & $[ \fr{m}_{13}, \fr{m}_{13}] \subset \fr{m}_{1} \oplus \fr{h},$ & $[\fr{m}_{23}, \fr{m}_{23}] \subset \fr{m}_{2} \oplus \fr{h},$\\
 $[ \fr{m}_{12}, \fr{m}_{23}] \subset \fr{m}_{13},$ & $[ \fr{m}_{13}, \fr{m}_{23}] \subset \fr{m}_{12},$ & $[\fr{m}_{12}, \fr{m}_{13}] \subset \fr{m}_{23}.$ 
\end{tabular}
\end{center}
 
  Thus  we see that the only non zero triplets (up to permutation of indices) are 
$${1 \brack {11}}, \  {2 \brack {22}}, \    {(12) \brack {1(12)}}, \   {(13) \brack {1(13)}},\  
   {(12) \brack {2(12)}}, \    {(23) \brack {2(23)}},  \   {(13) \brack {(12)(23)}}, 
$$
where  $\displaystyle{{i \brack {i i}} }$ is non zero only for $k_1, k_2 \geq 3$. 
\begin{lemma}\label{lemma5.1}
The components  of  the Ricci tensor ${r}$ for the $\Ad(\SO(k_1)\times\SO(k_2)\times\SO(k_3))$-invariant scalar product $ \langle \  \ ,\ \ \rangle $ on $G/H$ defined by  \em{(\ref{metric1})} are given as follows:  
{\small 
\begin{equation*}
%\left. 
\begin{array}{lll} 
r_1 &= & \displaystyle{\frac{1}{2 x_1} +
\frac{1}{4 d_1 } \biggl({1 \brack {11}}\frac{1}{x_{1}} +{1 \brack {(12)(12)}}  \frac{x_1}{{x_{12}}^2}+{1 \brack {(13)(13)}}  \frac{x_1}{{x_{13}}^2} \biggr)} \\ \\ & & 
 \displaystyle{- \frac{1}{2 d_1 } \biggl({1 \brack {11}}\frac{1}{x_{1}} +{(12) \brack {1(12)}}  \frac{1}{{x_{1}}}+{(13) \brack {1(13)}}  \frac{1}{{x_{1}}} \biggr),} 
 \\   \\
 r_2 &= & 
 \displaystyle{\frac{1}{2 x_2} +
\frac{1}{4 d_2 } \biggl({2 \brack {22}}\frac{1}{x_{2}} +{2 \brack {(12)(12)}}  \frac{x_2}{{x_{12}}^2}+{2 \brack {(23)(23)}}  \frac{x_2}{{x_{23}}^2} \biggr)} \\ \\ & & 
 \displaystyle{- \frac{1}{2 d_2 } \biggl({2 \brack {22}}\frac{1}{x_{2}} +{(12) \brack {2(12)}}  \frac{1}{{x_{2}}}+{(23) \brack {2(23)}}  \frac{1}{{x_{2}}} \biggr),} 
 \end{array} 
\end{equation*} }
{\small 
\begin{equation*}%\label{eq13}
%\left. 
\begin{array}{lll} 
r_{12} &= &  \displaystyle{\frac{1}{ 2 x_{12}} +\frac{1}{4 d_{12}}\biggl({(12) \brack {1 (12)}}  \frac{1}{x_{1}}\times 2 + {(12) \brack {2 (12)}}  \frac{1}{x_{2}} \times2+{(12) \brack {(13) (23)}}   \frac{x_{12}}{x_{13} x_{23} } \times 2 \biggr) } \\  \\ & & 
\displaystyle{-\frac{1}{ 2  d_{12}} \biggl( {1 \brack {(12) (12)}} \frac{x_1}{{x_{12}}^2} + {(12) \brack {(12) 1}} \frac{1}{x_{1}}+ {2 \brack {(12) (12)}}\frac{x_2}{{x_{12}}^2} +{(12) \brack {(12) 2}} \frac{1}{x_{2}} }  \\ \\ & & \displaystyle{+{(13) \brack {(12) (23)}}   \frac{x_{13}}{x_{12} x_{23} }+{(23) \brack {(12) (13)}}   \frac{x_{23}}{x_{12} x_{13} } \biggr) }, 
\\
r_{13}  &= &    \displaystyle{\frac{1}{ 2 x_{13}} +\frac{1}{4 d_{13}}\biggl({(13) \brack {1 (13)}}  \frac{1}{x_{1}}\times 2 + {(13) \brack {2 (13)}}  \frac{1}{x_{2}} \times2+{(13) \brack {(12) (23)}}   \frac{x_{13}}{x_{12} x_{23} } \times 2 \biggr) } \\  \\ & & 
\displaystyle{-\frac{1}{ 2  d_{13}} \biggl( {1 \brack {(13) (13)}} \frac{x_1}{{x_{13}}^2} + {(13) \brack {(13) 1}} \frac{1}{x_{1}} +{(12) \brack {(13) (23)}}   \frac{x_{12}}{x_{13} x_{23} }+{(23) \brack {(13) (12)}}   \frac{x_{23}}{x_{13} x_{12} } \biggr) }, 
\\  \\
r_{23}  &= &    \displaystyle{\frac{1}{ 2 x_{23}} +\frac{1}{4 d_{23}}\biggl({(23) \brack {2 (23)}}  \frac{1}{x_{2}}\times 2 + {(23) \brack {3 (23)}}  \frac{1}{x_{3}} \times2+{(23) \brack {(12) (13)}}   \frac{x_{23}}{x_{12} x_{13} } \times 2 \biggr) } \\  \\ & & 
\displaystyle{-\frac{1}{ 2  d_{23}} \biggl( {2 \brack {(23) (23)}} \frac{x_2}{{x_{23}}^2} + {(23) \brack {(23) 2}} \frac{1}{x_{2}} +{(12) \brack {(23) (13)}}   \frac{x_{12}}{x_{23} x_{13} }+{(13) \brack {(23) (12)}}   \frac{x_{13}}{x_{23} x_{12} } \biggr) }, 
\end{array} 
%\right\}
\end{equation*} }
where $n = k_1+k_2+k_3$. 
\end{lemma}
%%%%%%%%%%%%%%%
%
%20130529%
%
%20130623%
%
%
Now we recall some Lemmas from  \cite{ADN1} giving some more details in the proofs. 

\begin{lemma}\label{lemma5.2aa}  Let $\fr{q}$ be a simple subalgebra of $\fr{so}(n)$.  Consider an orthonormal basis $\{ f_j \}$  of $\fr{q}$ 
 with respect to $-B$ {\em({\em negative of the Killing form of $\fr{so}(n)$})} and denote by
 $B_{\fr{q}}$ the Killing form of $\fr{q}$.  Then,  for $i = 1, \ldots, \dim \fr{q}$, we have 
\begin{equation*}\label{eq14a}
\sum_{j, k = 1}^{\dim \fr{q}} \left(-B([f_i, f_j], f_k \right)^2 = {\al}^{\fr{q}}_{\fr{so}(n)},
\end{equation*}
where ${\al}^{\fr{q}}_{\fr{so}(n)}$ is the constant determined by $B_{\fr{q}} = {\al}^{\fr{q}}_{\fr{so}(n)}\cdot B|_{\fr{q}}$. 
\end{lemma} 
 \begin{proof}
 Note that the vectors $\left\{ \widetilde{f_j} = \left(1/\sqrt{{\al}^{\fr{q}}_{\fr{so}(n)}}\right) f_j \right\}$ form an orthonormal basis of $\fr{q}$ with respect to $-B_{\fr{q}}$. For $X, Y \in \fr{q}$,  we have

% \begin{equation*}
% \begin{array}{l} 
% \displaystyle 
% -B_{\fr{q}}(X, Y) = -\tr (\ad(X) \ad(Y)) = -\sum_{j} ( - B_{\fr{q}}(\ad(X) \ad(Y) \widetilde{f_j} , \widetilde{f_j} ) ) \\ 
% \displaystyle = -\sum_{j} (  B_{\fr{q}}( \ad(Y) \widetilde{f_j} , \ad(X) \widetilde{f_j} ) ) = 
%  1/{\al}^{\fr{q}}_{\fr{so}(n)}\cdot  \sum_{j} (-B_{\fr{q}}) ( \ad(X) {f_j} ,  \ad(Y) {f_j} )   \\ 
% \displaystyle   = \sum_{j} -B ( \ad(X) {f_j} ,  \ad(Y) {f_j} ) %1/{\al}^{\fr{q}}_{\fr{so}(n)}\cdot  \sum_{j, k } (-B_{\fr{q}} )( \ad(X) {f_j} ,  \widetilde{f_k})  (-B_{\fr{q}} )( \ad(Y) {f_j},  \widetilde{f_k} ) \\
% \displaystyle =  \sum_{j,  \, k } (-B( \ad(X) {f_j} ,  {f_k}) ) \cdot (-B( \ad(Y) {f_j},  {f_k} )). 
%\end{array}
%\end{equation*}
\begin{eqnarray*}
-B_{\fr{q}}(X, Y) &=& -\tr (\ad(X) \ad(Y)) = -\sum_{j} ( - B_{\fr{q}}(\ad(X) \ad(Y) \widetilde{f_j} , \widetilde{f_j} ) ) \\
\displaystyle &=& -\sum_{j} (  B_{\fr{q}}( \ad(Y) \widetilde{f_j} , \ad(X) \widetilde{f_j} ) ) = 
  1/{\al}^{\fr{q}}_{\fr{so}(n)}\cdot  \sum_{j} (-B_{\fr{q}}) ( \ad(X) {f_j} ,  \ad(Y) {f_j} )   \\ 
\displaystyle &=& \sum_{j} -B ( \ad(X) {f_j} ,  \ad(Y) {f_j} ) %1/{\al}^{\fr{q}}_{\fr{so}(n)}\cdot  \sum_{j, k } (-B_{\fr{q}} )( \ad(X) {f_j} ,  \widetilde{f_k})  (-B_{\fr{q}} )( \ad(Y) {f_j},  \widetilde{f_k} ) \\
 \displaystyle =  \sum_{j,  \, k } (-B( \ad(X) {f_j} ,  {f_k}) ) \cdot (-B( \ad(Y) {f_j},  {f_k} )). 
\end{eqnarray*}

In particular, we  have 
$$ 1=  -B_{\fr{q}}( \widetilde{f_i}, \widetilde{f_i}) =  1/{\al}^{\fr{q}}_{\fr{so}(n)}( -B_{\fr{q}})( {f_i}, {f_i})  = 1/{\al}^{\fr{q}}_{\fr{so}(n)} \sum_{j,\,  k = 1 }^{\dim \fr{q}}  (-B( [ {f_i}, {f_j}] ,  {f_k}) )^2. 
$$ 
 \end{proof}

\begin{lemma}\label{lemma5.20} For $a, b, c = 1, 2, 3$ and $(a - b)(b - c) (c - a) \neq 0$ the following relations hold:
\begin{equation*}\label{eq14}
\begin{array}{lll} 
 \displaystyle{{a \brack {a a}} = \frac{k_a (k_a -1)(k_a -2)}{2 (n -2)} },   &  \displaystyle{{a \brack {(a b) (a b)}} = \frac{k_a  k_b (k_a -1)}{2 (n -2)} }, &  \displaystyle{{(a c) \brack {(a b ) (b c)}} = \frac{k_a  k_b  k_c}{2 (n -2)} }. 
\end{array} 
\end{equation*}
\end{lemma} 
 \begin{proof} For the standard Lie subalgebra $\fr{so}(k) \subset \fr{so}(n)$ we have $\displaystyle  {\al}^{\fr{so}\, (k)}_{\fr{so} \, (n)} = \frac{k-2}{n-2}$ (\cite{DZ}).  Note that for $a =1, 2, 3$ it is $\fr{m}_a = \fr{so}(k_a)$  and  $\displaystyle{{a \brack a a}  =\sum_{i, j, k = 1}^{\dim  \fr{so} \, (k_a)} \left(-B([f_i, f_j], f_k \right)^2 }$. Thus 
 from Lemma \ref{lemma5.2aa} we obtain that 
 $$ \displaystyle{{a \brack {a a}} = {\al}^{\fr{so}\, (k_\al)}_{\fr{so} \, (n)}\dim \fr{so}(k_\al)
 = \frac{k_a (k_a -1)(k_a -2)}{2 (n -2)} }.$$  
 
 For the second relation we consider the Lie subalgebra $\fr{so}(k_a +k_b) \subset \fr{so}(n)$. 
 Recall the relations 
\begin{equation}\label{brakso}
   [ \fr{m}_a,  \fr{m}_a] =  \fr{m}_a, \, \,   [ \fr{m}_a,  \fr{m}_b] =(0),  \, \,  [\fr{m}_a, \fr{m}_{a b}] \subset\fr{m}_{a b},   
\end{equation} 
and,  for $i = 1, \ldots, \dim \fr{m}_a$, Lemma \ref{lemma5.2aa} implies that
\begin{equation*}\label{eq15a}
{\al}^{\fr{so}(k_a+k_b)}_{\fr{so}(n)}=\sum_{j, k = 1}^{\dim \fr{so} \, (k_a+ k_b)} \left(-B([f_i, f_j], f_k \right)^2 =  \frac{k_a+ k_b -2}{n-2}. 
\end{equation*}
Take a basis $\{f_i\}$ of $\fr{m}_a$.  Then by taking into account 
(\ref{brakso}) we obtain that
\begin{equation*}
\begin{array}{lll}
\displaystyle{\sum_{i=1}^{\dim\fr{m}_a}\left(\sum_{j,k=1}^{\dim\fr{so}(k_a+k_b)} \left(-B([f_i, f_j], f_k \right)^2\right) }&=&
\displaystyle{{a \brack {a a}} + 0 + {(a b) \brack {a (a b)}}} 
\end{array}
\end{equation*}
and  thus 
\begin{equation*}
\begin{array}{lll}
\displaystyle{{a \brack {a a}} +  {(a b) \brack {a (a b)}}} &=&\dim\fr{m}_a\cdot  {\al}^{\fr{so}\, (k_a+k_b)}_{\fr{so}\, (n)} = 
\displaystyle{\dim  \fr{so}(k_a) \cdot  \frac{k_a+ k_b -2}{n-2} = \frac{k_a (k_a-1)( k_a+ k_b -2)}{2(n-2)} },
\end{array}
\end{equation*}

hence we obtain that 
$$\displaystyle{ {(a b) \brack {a (a b)}} =  \frac{k_a (k_a-1) k_b  }{2(n-2)} }. 
$$ 
Finally, for the third relation we note that 
$$ [ \fr{m}_{a b},  \fr{m}_a] =  \fr{m}_{a b}, \, \,  [ \fr{m}_{a b},  \fr{m}_b] =  \fr{m}_{a b},  \, \,  [ \fr{m}_{a b},  \fr{m}_{a b}] \subset  \fr{m}_{a}  + \fr{m}_{b}  \, \,  [ \fr{m}_{a b},  \fr{m}_{a c}] \subset  \fr{m}_{b c },\, \,  [ \fr{m}_{a b},  \fr{m}_{b c}] \subset  \fr{m}_{a c}, 
$$
 and, for $i = 1, \ldots, \dim \fr{m}_{a b}$, 
\begin{equation*}\label{eq16a}
\sum_{j, k = 1}^{\dim \fr{so} \, (k_a+ k_b+ k_c)} \left(-B([f_i, f_j], f_k \right)^2 = 1. 
\end{equation*}
Then we see that   
$$\displaystyle{  {(a b) \brack {(a b)  a}} +  {(a b) \brack { (a b) b }} + {a \brack {(a b)  (a b)}} +  {b \brack { (a b) (a b) }}+ {( b c) \brack {(a b)  (a c)}} +  {(a c)  \brack { (a b) (b c) }} = \dim \fr{m}_{a b}  = k_a  k_b. 
}
$$
Thus we obtain that 
$$\displaystyle{ {(a c ) \brack {(a b) (b c)}} =  \frac{k_a k_b  k_c  }{2(n-2)} }. 
$$ 
 \end{proof}
 
%%%%%%%% 

By using Lemma \ref{lemma5.20} the components of the Ricci tensor given in Lemma \ref{lemma5.1} are explicitly given as follows:
\begin{lemma}\label{lemma5.4}
The components  of  the Ricci tensor ${r}$ for the invariant metric $ \langle \   ,\  \rangle $ on $G/H$ defined by  \em{(\ref{metric1})} are given as follows:  
\begin{equation}\label{eq18}
\left. 
\small{\begin{array}{lll} 
r_1 &=&  \displaystyle{\frac{k_1-2}{4 (n -2)  x_1} +
\frac{1}{4 (n -2) } \biggl(k_2 \frac{x_1}{{x_{12}}^2}} +k_3 \frac{x_1}{{x_{13}}^2} \biggr), 
 \\   \\
 r_2 &=&  
\displaystyle{\frac{k_2-2}{4 (n -2)  x_2} +
\frac{1}{4 (n -2)} \biggl(k_1 \frac{x_2}{{x_{12}}^2} +k_3 \frac{x_2}{{x_{23}}^2} \biggr),} 
\\  \\
r_{12} &=&   \displaystyle{\frac{1}{ 2 x_{12}} +\frac{k_3}{4 (n -2)}\biggl(\frac{x_{12}}{x_{13} x_{23}} - \frac{x_{13}}{x_{12} x_{23}} - \frac{x_{23}}{x_{12} x_{13}}\biggr) }\\  \\ & &
\displaystyle{-
\frac{1}{4 (n -2)} \biggl( (k_1-1) \frac{x_1}{{x_{12}}^2} + (k_2-1) \frac{x_2}{{x_{12}}^2} \biggr)},
\\  \\
r_{13}  &=&   \displaystyle{\frac{1}{  2 x_{13}} +\frac{k_2}{4 (n -2)}\biggl(\frac{x_{13}}{x_{12} x_{23}} - \frac{x_{12}}{x_{13} x_{23}} - \frac{x_{23}}{x_{12} x_{13}}\biggr) -
\frac{1}{4 (n -2)} \biggl( (k_1-1) \frac{x_1}{{x_{13}}^2}  \biggr)}
\\ \\
r_{23}  &=&   \displaystyle{\frac{1}{ 2 x_{23}} +\frac{k_1}{4 (n -2)}\biggl(\frac{x_{23}}{x_{13} x_{12}} - \frac{x_{13}}{x_{12} x_{23}} - \frac{x_{12}}{x_{23} x_{13}}\biggr) -
\frac{1}{4 (n -2)} \biggl(  (k_2-1) \frac{x_2}{{x_{23}}^2} \biggr)}.   
\end{array} } \right\}
\end{equation}
where $n = k_1+k_2+k_3$. 
\end{lemma}

%%%%%%%%%%

For  $k_1 =1$  we have  the Stiefel manifold $G/H=\SO(1+k_2+k_3)/\SO(k_3)$ with corresponding decomposition 
$$\fr{p} =   \fr{so}(k_2)\oplus  \fr{m}_{12} \oplus  \fr{m}_{13} \oplus  \fr{m}_{23}. 
$$
We consider
$G$-invariant metrics on $G/H$ determined by the  $\Ad(\SO(k_2)\times\SO(k_3))$-invariant scalar products on $\fr{p}$ given by 

 \begin{equation} \label{metric2} 
 \langle \   ,\ \rangle =     x_2 \, (-B) |_{ \fr{so}(k_2)} + x_{12} \,  (-B) |_{ \fr{m}_{12}} + x_{13} \,  (-B) |_{ \fr{m}_{13}} + x_{23} \,  (-B) |_{ \fr{m}_{23}}
\end{equation}

\begin{lemma}\label{ricV4Rn}
The components  of  the Ricci tensor ${r}$ for the invariant metric $ \langle \  ,\  \rangle $ on $G/H$ defined by  \em{(\ref{metric2})}, are given as follows:  
\begin{equation}\label{eq19}
\left. {\small \begin{array}{l} 
 r_2  = 
\displaystyle{\frac{k_2-2}{4 (n -2) x_2} +
\frac{1}{4 (n -2)} \biggl(  \frac{x_2}{{x_{12}}^2} +k_3 \frac{x_2}{{x_{23}}^2} \biggr)},
\\  \\
r_{12}   =  \displaystyle{\frac{1}{2 x_{12}} +\frac{k_3}{4 (n -2)}\biggl(\frac{x_{12}}{x_{13} x_{23}} - \frac{x_{13}}{x_{12} x_{23}} - \frac{x_{23}}{x_{12} x_{13}}\biggr) -
\frac{1}{4 (n -2)} \biggl(  (k_2-1) \frac{x_2}{{x_{12}}^2} \biggr)},
\\  \\
r_{23}   =  \displaystyle{\frac{1}{2 x_{23}} +\frac{1}{4 (n -2)}\biggl(\frac{x_{23}}{x_{13} x_{12}} - \frac{x_{13}}{x_{12} x_{23}} - \frac{x_{12}}{x_{23} x_{13}}\biggr) -
\frac{1}{4 (n -2)} \biggl(  (k_2-1) \frac{x_2}{{x_{23}}^2} \biggr)}, 
\\  \\
r_{13}   =  \displaystyle{\frac{1}{ 2 x_{13}} +\frac{k_2}{4 (n -2)}\biggl(\frac{x_{13}}{x_{12} x_{23}} - \frac{x_{12}}{x_{13} x_{23}} - \frac{x_{23}}{x_{12} x_{13}}\biggr) }. 
\end{array} } \right\} 
\end{equation}
where $n = 1+k_2+k_3$. 
\end{lemma}

%%%%%%%%Sept1%%%%
 \section{The Stiefel manifold $V_4\mathbb{R} ^n$}
For the Stiefel manifold $V_4\mathbb{R}^n = \SO(n)/\SO(n-4)$ we let $k_2=3$ and $k_3=n-4$ ($n\ge 6$) and consider  $\Ad( \SO(3)\times\SO(n-4))$-invariant scalar products of the form (\ref{metric2}). 
\begin{theorem}\label{theorem_v4_R^n} 
The Stiefel manifold $V_4\mathbb{R}^n = \SO(n)/\SO(n-4)$ ($n\ge 6$)  admits at least four invariant Einstein metrics. Two of them are Jensen's  metrics and the other two are given by $\Ad( \SO(3)\times\SO(n-4))$-invariant scalar products of the form (\ref{metric2}). 
\end{theorem} 

 \begin{proof} 
 From Lemma \ref{ricV4Rn} we see that  the components  of  the Ricci tensor ${r}$ for such an invariant metric are given by 

 \begin{equation}\label{eq21}
\left. {\small \begin{array}{l} 
 r_2  = 
\displaystyle{\frac{1}{4(n-2) x_2} +
\frac{1}{4(n-2)} \biggl(  \frac{x_2}{{x_{12}}^2} +(n-4) \frac{x_2}{{x_{23}}^2} \biggr)},
\\  \\
r_{12}   =  \displaystyle{\frac{1}{2 x_{12}} +\frac{n-4}{4(n-2)}\biggl(\frac{x_{12}}{x_{13} x_{23}} - \frac{x_{13}}{x_{12} x_{23}} - \frac{x_{23}}{x_{12} x_{13}}\biggr) -
\frac{1}{2(n-2)}  \frac{x_2}{{x_{12}}^2}},
\\  \\
r_{23}   =  \displaystyle{\frac{1}{2 x_{23}} +\frac{1}{4(n-2)}\biggl(\frac{x_{23}}{x_{13} x_{12}} - \frac{x_{13}}{x_{12} x_{23}} - \frac{x_{12}}{x_{23} x_{13}}\biggr) -
\frac{1}{2(n-2)} \frac{x_2}{{x_{23}}^2}},
\\  \\
r_{13}   =  \displaystyle{\frac{1}{ 2 x_{13}} +\frac{3}{4(n-2)}\biggl(\frac{x_{13}}{x_{12} x_{23}} - \frac{x_{12}}{x_{13} x_{23}} - \frac{x_{23}}{x_{12} x_{13}}\biggr) }.
\end{array} } \right\} 
\end{equation}
 We consider the system of equations    
 \begin{equation}\label{eq21a} 
 r_2 =  r_{12}, \,  r_{12} = r_{13}, \,  r_{13} = r_{23}. 
 \end{equation}

Then finding Einstein metrics of the form (\ref{metric2})  reduces  to finding positive solutions of  system (\ref{eq21a}),  and  we normalize  our equations by putting $x_{23}=1$. Then we have the system of equations: 
 \begin{equation}\label{eq21b}
\left. { \begin{array}{lll}
f_1&  = & -(n-4) {x_{12}}^3 x_{2}+(n-4) {x_{12}}^2
   {x_{13}} x_{2}^2+(n-4) {x_{12}} {x_{13}}^2
   x_{2} \\
   & & -2 (n-2) {x_{12}} {x_{13}} x_{2} +(n-4)
   {x_{12}} x_{2}+{x_{12}}^2 {x_{13}}+3 {x_{13}}
   x_{2}^2=0, \\
f_2 &= &(n-3) {x_{12}}^3-2 (n-2) {x_{12}}^2
   {x_{13}}-(n-5) {x_{12}} {x_{13}}^2 \\
   & & +2 (n-2)
   {x_{12}} {x_{13}}+(3-n) {x_{12}}+2 {x_{12}}^2
   {x_{13}} x_{2}-2 {x_{13}} x_{2}=0, \\
f_3& =& (n-2)
   {x_{12}} {x_{13}}-(n-2)
   {x_{12}}+{x_{12}}^2-{x_{12}} {x_{13}} x_{2}-2
   {x_{13}}^2+2 =0. 
\end{array} } \right\} 
\end{equation}
 We consider a polynomial ring $R= {\mathbb Q}[z,  x_2, x_{12}, x_{13}] $ and an ideal $I$ generated by 
$\{ f_1, \, f_2, \, f_3, $  $  \,z  \, x_2 \, x_{12} \, x_{13} -1\}  
$  to find non-zero solutions of equations (\ref{eq21b}). 
We take a lexicographic order $>$  with $ z >   x_2 > x_{12} > x_{13}$ for a monomial ordering on $R$. Then, by the aid of computer, we see that a  Gr\"obner basis for the ideal $I$ contains the  polynomial
$(x_{13} - 1) \, h_{1}(x_{13}),$
where $h_{1}(x_{13})$ is a polynomial of   $x_{13}$  given by 
{ \small 
 \begin{equation*}%\label{eq222}
\begin{array}{l}
h_{1}(x_{13})   =  (n-1)^3 (5 n-11)^2 \left(n^3-10 n^2+33 n-35\right)
   \left(n^3-6 n^2+9 n-3\right) {x_{13}}^{10} \\
    -2 (n-1)^2 (5 n-11) \left(17
   n^8-356 n^7+3221 n^6-16396 n^5+51159 n^4-99720 n^3 \right. \\
  \left.  +117862
   n^2-76568 n+20649\right) {x_{13}}^9+(n-1) \left(4
   n^{11}+389 n^{10}-11430 n^9+136940 n^8 \right. \\
 \left.   -946084 n^7+4220820
   n^6-12735744 n^5+26330445 n^4-36830352 n^3+33361745
   n^2 \right. \\
   \left. -17678114 n+4164053\right) {x_{13}}^8 -4
   \left(8 n^{12}-38 n^{11}-2320 n^{10}+43360 n^9-379590
   n^8 \right. 
   \\
  \left. 
   +2055155 n^7-7507061 n^6+19112638 n^5-34063584
   n^4+41706995 n^3-33417851 n^2 \right.
  \end{array} 
\end{equation*} }
{\small 
\begin{equation*}%\label{eq13}
%\left. 
\begin{array}{lll} 
     \left. +15765962 n-3316050\right)
   {x_{13}}^7+\left(112 n^{12}-2718 n^{11}+27906
   n^{10}-149523 n^9+354855 n^8 \right. \\
  \left. +588726 n^7-7694150
   n^6+29295831 n^5-65164167 n^4+92342878 n^3-82220114
   n^2 \right. \\
%    \end{array}  \end{equation*} }
%{   \small 
% \begin{equation*}
%\begin{array}{l}
  \left. +41992646 n-9373722\right) {x_{13}}^6-2 \left(112
   n^{12}-3338 n^{11}+45506 n^{10}-376557 n^9 \right. \\
  \left. +2113393
   n^8-8496684 n^7+25132832 n^6-55172371 n^5+89317711
   n^4-104159676 n^3 \right. \\
  \left. +83190848 n^2-40884390 n+9337014\right)
   {x_{13}}^5+\left(280 n^{12}-8710 n^{11}+123662
   n^{10}  \right. \\
  \left. -1060617 n^9+6124653 n^8-25086974 n^7+74662934
   n^6-162341127 n^5+255246159 n^4 \right. \\
  \left. -282268554 n^3+208035522
   n^2-91729890 n+18337990\right) {x_{13}}^4-4 \left(56
   n^{12}-1710 n^{11}  \right. \\
  \left. +23600 n^{10}-194131 n^9+1056185
   n^8-3982619 n^7+10582237 n^6-19666327 n^5 \right. \\
  \left. +24629929
   n^4-18903391 n^3+6556083 n^2+985682 n-1096346\right)
   {x_{13}}^3 \\
  +(n-1)
   \left(112 n^{11}-3115 n^{10}+38156 n^9-268869 n^8   +1189262
   n^7-3348224 n^6\right. \\
 \left.+5627178 n^5 -3967104 n^4-3831854
   n^3+11143963 n^2  -9643014 n+3094229\right) {x_{13}}^2   \\
-2 (n-5) (n-3) (n-1)^2 (n+1) \left(16 n^7-275
   n^6+1868 n^5-6039 n^4+7372 n^3 +7943 n^2 \right. \\
 \left.-31120
   n+23163\right) {x_{13}} +(n-5)^2
   (n-3)^2 (n-1)^3 (n+1)^2 \left(4 n^3-23 n^2-10
   n+161\right). 
\end{array}  \end{equation*} }
If $ x_{13} =1$ we see that $ f_3 = x_{12}( x_{12} - x_2) =0$ and thus the system of equations (\ref{eq21b}) reduces to the system of  equations 
$$ x_{12} = x_2, \quad (n-1) {x_2}^2 - 2 (n-2){x_2} + 2 =0. $$
Thus we obtain  two solutions for the system  (\ref{eq21b}), namely  
$$x_{12} = x_2 = (n-2-\sqrt{n^2-6 n+6})/(n-1),  \, \, x_{13} = x_{23} =1 $$
and
$$
x_{12} = x_2 = (n-2+\sqrt{n^2-6 n+6})/(n-1),  \, \, x_{13} = x_{23} =1. 
$$
These are known as Jensen's Einstein metrics on Stiefel manifolds.

If $ x_{13} \ne 1$ then $ h_{1}(x_{13}) = 0$ and we claim that the equation  $ h_{1}(x_{13}) = 0$ has at least two positive roots. 
Firstly we consider the value $ h_{1}(x_{13})$ at $ x_{13} =1$. We have 
{\small  \begin{equation*}
\begin{array}{l}
h_{1}(1)   = -800 \left(6 n^5-88 n^4+476 n^3-1175 n^2+1274 n-490\right) \\
= -800 \left(6 (n-6)^5+92 (n-6)^4+524 (n-6)^3+1345 (n-6)^2+1430
   (n-6)+278 \right). 
\end{array}  \end{equation*} }
Then we see that $h_{1}(1) < 0$ for $ n \geq 6$. 

Secondly we consider the value  $ h_{1}(x_{13})$ at $ x_{13} =0$. Then  we have  
{\small  \begin{equation*}
\begin{array}{l}
h_{1}(0)   = (n-5)^2 (n-3)^2 (n-1)^3 (n+1)^2 \left(4 n^3-23 n^2-10 n+161\right)
\\
=  (n-5)^2 (n-3)^2 (n-1)^3 (n+1)^2 \left(4 (n-5)^3+37 (n-5)^2+60 (n-5)+36\right). 
\end{array}  \end{equation*} }

Thirdly we consider the value $ h_{1}(x_{13})$ at $ x_{13} =2$.  We have  
{\small  \begin{equation*}
\begin{array}{l}
h_{1}(2)   = (n-5) \left(4 n^{11}+313 n^{10}-2902 n^9-11175 n^8+334728 n^7-2555222 n^6 \right. \\
\left. +10151316 n^5-22397134 n^4+25374596 n^3-8599331 n^2-8372942  n+6279733\right)
\\
= (n-5)\left( 4 (n-5)^{11}+533 (n-5)^{10}+18248 (n-5)^9+292860 (n-5)^8 \right. \\
+2795928 (n-5)^7+17723008 (n-5)^6+77831856 (n-5)^5+236200016 (n-5)^4 \\
\left.  +477068416 (n-5)^3+593616384 (n-5)^2+389736448 (n-5)+89941248 \right). 
\end{array}  \end{equation*} }
Thus we see that $h_{1}(0) >0$ and $h_{1}(2) > 0$ for $ n \geq 6$.  Hence, we obtain two solutions $ x_{13}= \alpha_{13}, \beta_{13}$ of the equation  $ h_{1}(x_{13}) = 0$ between $ 0 < \alpha_{13} < 1$ and  $1 < \beta_{13} < 2$. 

 We consider a polynomial ring $R= {\mathbb Q}[z,  x_2, x_{12}, x_{13}] $ and an ideal $J$ generated by 
$\{ f_1, \, f_2, \, f_3, $  $  \,z  \, x_2 \, x_{12} \, x_{13} \,( x_{13} - 1)  -1\}  
$ and  take a lexicographic order $>$  with $ z >   x_2 > x_{12} > x_{13}$ for a monomial ordering on $R$. Then, by the aid of computer, we see that a  Gr\"obner basis for the ideal $J$ contains the  polynomials
$ h_{1}(x_{13}) $ and 
 $$ 8 (n-3) (n-2)^3 (n-1)^2 a(n) x_{12} - w_{12}(x_{13}).$$
 Also, for the same ideal $J$  and the 
 lexicographic order $>$  with $ z >   x_{12} > x_2 > x_{13}$ for monomials on $R$,
 we see that a Gr\"obner basis for $J$ contains the polynomial
 $$8 (n-5) (n-2)^3 (n-1)^3 (n+1) \left(4 n^3-23 n^2-10 n+161\right)  a(n) x_2 -w_2 (x_{13}),$$
 where $a(n)$ is a polynomial of $n$ of degree 43 with integer  coefficients and  $ w_{12}(x_{13})$,  $ w_{2}(x_{13})$ are polynomials of $x_{13}$  (and $n$) with integer  coefficients. It is easy to check that $ a(n)  > 0$ for $ n \geq 6$. Thus for the positive values $ x_{13}= \alpha_{13}, \beta_{13}$ found above  we obtain real values $ x_{2}= \alpha_{2}, \beta_{2}$ and 
 $ x_{12}= \alpha_{12}, \beta_{12}$ as solutions of  the system of equations (\ref{eq21b}).
 
 We claim that $ \alpha_{2}, \beta_{2}, \alpha_{12}, \beta_{12}$ are positive.  We consider the ideal $J$ generated by 
$\{ f_1, \, f_2, \, f_3, $  $  \,z  \, x_2 \, x_{12} \, x_{13} \,( x_{13} - 1)  -1\}$ and now  take  a lexicographic order $>$  with $ z >   x_2 > x_{13} > x_{12}$ for a monomial ordering on $R$.  Then, by the aid of computer, we see that a  Gr\"obner basis for the ideal $J$ contains the  polynomial $h_{2}(x_{12})$,  where 
 the polynomial  $h_{2}(x_{12})$ can be written as 
{\small  \begin{equation*} 
\begin{array}{l}
h_{2}(x_{12})   = \left((n-6)^3+8
   (n-6)^2+21 (n-6)+19\right) \\  \left((n-6)^3+12
   (n-6)^2+45 (n-6)+51\right) (n-1)^3 {x_{12}}^{10} \\
    -2 \left((n-6)^6+28  (n-6)^5+294 (n-6)^4+1534 (n-6)^3+4277 (n-6)^2 \right. \\  \left. +6122
   (n-6)+3549\right) (n-2) (n-1)^2 {x_{12}}^9 \\ 
   + 2 (n-1)  \left(20 (n-6)^7+549 (n-6)^6+6340 (n-6)^5+39979
   (n-6)^4  \right. \\  \left. +148820 (n-6)^3+327348 (n-6)^2+394439
   (n-6)+201193\right) {x_{12}}^8 
  \\ 
    -4 (n-2) \left(2 (n-6)^7+139 (n-6)^6+2554
   (n-6)^5+22040 (n-6)^4+104177 (n-6)^3  \right. \\  \left.+278253
   (n-6)^2 +395233 (n-6)+232722\right) {x_{12}}^7 
  \\ 
   +\left(145 (n-6)^7+5403 (n-6)^6+77859 (n-6)^5+586017
   (n-6)^4+2536504 (n-6)^3  \right. \\  \left. +6381392 (n-6)^2+8697776
   (n-6)+4977456\right) {x_{12}}^6 \\ 
    -2 (n-2) \left(3 (n-6)^6+616
   (n-6)^5+12631 (n-6)^4+103314 (n-6)^3   \right. \\  \left.+412576
   (n-6)^2+804920 (n-6)+615288\right) 
   {x_{12}}^5 \\
   +2 \left(53 (n-6)^6+3344 (n-6)^5+52675
   (n-6)^4+372392 (n-6)^3   \right. \\  \left.+1350984 (n-6)^2+2463232
   (n-6)+1792224\right) {x_{12}}^4 \\ 
   -16  (n-2) \left(46 (n-6)^4+1212
   (n-6)^3+9261 (n-6)^2+28030 (n-6)+29688\right)
   {x_{12}}^3 \\
  +(64 \left(39
   (n-6)^4+683 (n-6)^3+4270 (n-6)^2+11459
   (n-6)+11239\right) {x_{12}}^2   \\
 -512 (n-2) \left(8
   (n-6)^2+53 (n-6)+89\right)  {x_{12}} 
   +640 \left(4 (n-6)^2+24 (n-6)+37\right). 
\end{array}  
\end{equation*} }
Then  we see that, for $n \geq 6$, the coefficients of the polynomial  $h_{2}(x_{12})$ are positive  for  even degree terms and negative for odd degree terms. Thus  if  the equation $h_{2}(x_{12})  = 0$ has real solutions, then these are all  positive. In particular, $ x_{12}= \alpha_{12}, \beta_{12}$ are positive. 

%%%%%%%
Now we  take  a lexicographic order $>$  with $ z >   x_{12} > x_{13} > x_{2}$ for a monomial ordering on $R$.  Then, by the aid of computer, we see that a  Gr\"obner basis for the ideal $J$ contains the  polynomial $h_{3}(x_{2})$,  where 
 the polynomial  $h_{3}(x_{2})$ can be written as 
{\small 
 \begin{equation*}
\begin{array}{l}
h_{3}(x_{2})   = \left(4 (n-6)^3+49 (n-6)^2+146
   (n-6)+137\right) (n-1)^3 (5 n-11)^2 {x_{2}}^{10} \\
    -2 \left(4 (n-6)^5+148 (n-6)^4+1913 (n-6)^3+9717
   (n-6)^2+20631 (n-6)   \right.
   \\ \left. +15971\right) (n-2) (n-1)^2 (5
   n-11){x_{2}}^9 
   + 2 \left(2 (n-6)^9+284
   (n-6)^8+9007 (n-6)^7   \right. 
   \\ +139501 (n-6)^6  +1247253
   (n-6)^5  +6834591 (n-6)^4   +23266933 (n-6)^3 \\  \left. +47940391
   (n-6)^2+54761765 (n-6)+26704721\right) (n-1) {x_{2}}^8 
      \end{array} 
\end{equation*} }
{\small 
\begin{equation*}%\label{eq13}
%\left. 
\begin{array}{lll} 
 -4 \left(6 (n-6)^9+646 (n-6)^8+19688 (n-6)^7+289164
   (n-6)^6   \right. \\  \left.+2441215 (n-6)^5 +12719748 (n-6)^4+41647456
   (n-6)^3+83518006 (n-6)^2   \right. 
   \\  \left.+93810211
   (n-6)+45323972\right) (n-2) {x_{2}}^7 
\\   +\left(36 (n-6)^{10}+4609
   (n-6)^9+174677 (n-6)^8+3042442 (n-6)^7  \right. \\  \left.+30064190
   (n-6)^6+185689733 (n-6)^5+747015337
   (n-6)^4+1965616896 (n-6)^3   \right. \\  \left.+3272524240
   (n-6)^2+3136179408 (n-6)+1321501424\right) {x_{2}}^6 
 \\    -2 \left(1171 (n-6)^8+67984 (n-6)^7+1234216 (n-6)^6+11204886
   (n-6)^5  \right. 
\\
  \left.+59121913 (n-6)^4+190739546
   (n-6)^3+372812292 (n-6)^2  \right. 
 \\  \left.+407205824
   (n-6)+191474440\right) (n-2) 
   {x_{2}}^5 
 \\
   +2 \left(20229 (n-6)^8+596062 (n-6)^7+7566600 (n-6)^6+54209786
   (n-6)^5  \right. \\  \left.+240333483 (n-6)^4+676552424
   (n-6)^3+1182967080 (n-6)^2  \right. \\  \left.+1176294720
   (n-6)+509895520\right) {x_{2}}^4 \\ 
  -16 \left(4852 (n-6)^6+95667 (n-6)^5+777487
   (n-6)^4+3345377 (n-6)^3  \right. \\  \left.+8058895 (n-6)^2+10324814
   (n-6)+5503940\right) (n-2)
   {x_{2}}^3 
 \\
  +64 \left(917
   (n-6)^6+17831 (n-6)^5+143658 (n-6)^4+614607
   (n-6)^3  \right. \\  \left.+1474263 (n-6)^2+1881545 (n-6)+998839\right) {x_{2}}^2   \\
 -128
   \left(156 (n-6)^3+1344 (n-6)^2+3901
   (n-6)+3804\right) (n-3) (n-2) {x_{2}} \\
 +   640 \left(4 (n-6)^2+24 (n-6)+37\right) (n-3)^2. 
\end{array}  \end{equation*} }
Then  we see that, for $n \geq 6$, the coefficients of the polynomial  $h_{3}(x_{2})$ are positive  for  even degree terms and negative for odd degree terms. Thus  if  the equation $h_{3}(x_{2})  = 0$ has real solutions, then these are all  positive. In particular, the solutions $ x_{2}= \alpha_{2}, \beta_{2}$ are positive, thus we get our claim. 
Notice that the positive  solutions $ \{ x_{2} = \alpha_{2}, x_{12} = \alpha_{12}, x_{13} = \alpha_{13}, x_{23} = 1 \}$, 
$ \{ x_{2} = \beta_{2}, x_{12} = \beta_{12}, x_{13} = \beta_{13}, x_{23} = 1 \}$ satisfy $ \alpha_{13}, \beta_{13} \neq 1$.
Thus these solutions are different from   the Jensen's Einstein metrics on Stiefel manifolds. 
 \end{proof}

We can give  approximate values  of the invariant Einstein metrics on $V_4\bb R^n$: 

  \smallskip
  
 {\small 
%\begin{center}
 \begin{tabular}{|l|l|l|}
 \hline    
      scalar products    & invariant Einstein metrics    \\  
             \hline    
   $\Ad( \SO(3)\times\SO(n-4))$   &   $   x_{12} = x_2 = (n-2-\sqrt{n^2-6 n+6})/(n-1),  \, \, x_{13} = x_{23} =1$   \\ 
 -invariant scalar products  & $x_{12} = x_2 = (n-2+\sqrt{n^2-6 n+6})/(n-1),  \, \, x_{13} = x_{23} =1$  \\ 
  of the form (\ref{metric2})  &    ( Jensen's Einstein metrics ) \\
  & ----------------------------------------------------------------------- \\
    &  $(x_{23}, x_{13}, x_{12}, x_2)  =  (1, \alpha_{13}, \,\alpha_{12}, \, \alpha_{2})$ \\
       &  $(x_{23}, x_{13}, x_{12}, x_2)   =  ( 1, \beta_{13}, \, \beta_{12}, \, \beta_{2})$   \\
  \hline
 \end{tabular} }
 
 \medskip
 
 We see that for $ n \geq 9$,  we obtain $$1-2/n -6/n^2 <  \alpha_{13}  < 1-2/n -7/(2  n^2), \, \, 1 + 50/(63 n^2)  <  \beta_{13}  < 1 + 3/n^2, $$ 
 for $n \geq 7$, we obtain 
$$2 -2/n -6/n^2 <  \alpha_{12}  < 2 - 4/n - 31/(4 n^2),  \, \, 5/(3 n) +815/(162 n^2) <  \beta_{12}  < 5/(3 n) +10/( n^2)$$ 
and, for  $n \geq 16$, we obtain 
$$1 /(2 n) +  13/ (8 n^2) <  \alpha_{2}  <1 /(2 n) +  11/ (5 n^2),  \, \,  5/(9 n) +23/(20 n^2) <  \beta_{2}  < 5/(9 n) +10 / n^2.$$

%%%%%%%%%%%%%%%%%%%%20130923%%%

\section{The Stiefel manifold $V_5\mathbb{R} ^7$}
For the Stiefel manifold $V_5\mathbb{R} ^7 = \SO(7)/\SO(2)$ we let $k_1=2, k_2=3, k_3=2$ and
 consider  $\Ad(\SO(2)\times\SO(3)\times\SO(2))$-invariant scalar products of the form (\ref{metric1}) and
 for $k_1=1, k_2=4, k_3=2$ we consider $\Ad( \SO(4)\times\SO(2))$-invariant scalar products of the form (\ref{metric2}).

%%%%%%%%%%%%%%%%%%%%%%%%%

\begin{theorem}\label{theorem_v5_R^7} 
The Stiefel manifold $V_5\mathbb{R}^7 = \SO(7)/\SO(2)$  admits at least six invariant Einstein metrics. Two of them are Jensen's Einstein metrics,    the other two are given by $\Ad(\SO(2)\times\SO(3)\times\SO(2))$-invariant scalar products of the form (\ref{metric1}) and the rest two are given by $\Ad( \SO(4)\times\SO(2))$-invariant scalar products of the form (\ref{metric2}). 
\end{theorem} 
{\small 
\begin{center}
 \begin{tabular}{|l|l|l|}
 \hline
      scalar products    & approximate values    \\
                \thickline
   $\Ad(\SO(2)\times\SO(3)\times\SO(2))$  &  $(x_{23}, x_{13}, x_{12}, x_1, x_2)  = \left(1,  1,  ( 5+ \sqrt{7})/6, ( 5+ \sqrt{7})/6,  ( 5+ \sqrt{7})/6 \right)$  \\  -invariant scalar products  & $(x_{23}, x_{13}, x_{12}, x_1, x_2)  = \left(1, 1,  ( 5- \sqrt{7})/6, ( 5- \sqrt{7})/6,  ( 5- \sqrt{7})/6 \right)$ \\ of the form (\ref{metric1}) 
  &  ( Jensen's Einstein metrics )  \\ 
& ------------------------------------------------------------------ \\ 
  & $(x_{23}, x_{13}, x_{12}, x_1,  x_2)  \approx (1, 1.13934, \,0.620201, \, 0.831771, \, 0.149407 ) $ \\ & $ (x_{23}, x_{13}, x_{12},  x_1,  x_2)  \approx ( 1, 0.350124, \, 1.03223, \, 0.455639, \, 0.121264)$ \\ 
    \hline 
   $\Ad( \SO(4)\times\SO(2))$   &  $(x_{23}, x_{13}, x_{12}, x_2)  = \left(1,  1,  ( 5+ \sqrt{7})/6, ( 5+ \sqrt{7})/6 \right)$   \\ 
 -invariant scalar products  & $(x_{23}, x_{13}, x_{12}, x_2)  = \left( 1, 1,  ( 5- \sqrt{7})/6, ( 5- \sqrt{7})/6 \right)$  \\ 
  of the form (\ref{metric2})  &    ( Jensen's Einstein metrics ) \\
  & ------------------------------------------------------------------ \\
    &  $(x_{23}, x_{13}, x_{12}, x_2)  \approx (1, 0.253386, \,1.01652, \, 0.245146)$ \\
       &  $(x_{23}, x_{13}, x_{12}, x_2)  \approx ( 1, 1.16137, \, 0.669071, \, 0.291175)$   \\
  \hline
 \end{tabular}
 \end{center} }
 
 \begin{proof} 
 
 \underline{Case 1.} 
 $\Ad(\SO(2)\times\SO(3)\times\SO(2))$-invariant products of the form (\ref{metric1})

 From Lemma \ref{lemma5.4} we see that  the components  of  the Ricci tensor ${r}$ for the invariant metric are given by 
{ \small 
 \begin{equation}\label{eq29}
\left. 
\begin{array}{lll} 
r_1 & = &  \displaystyle{
\frac{1}{20} \biggl(3 \frac{x_1}{{x_{12}}^2}} +2 \frac{x_1}{{x_{13}}^2} \biggr), 
\quad    
 r_2  \, \, \,  =  \, \, \, 
\displaystyle{\frac{1}{ 20 x_{2}} +
\frac{1}{10} \biggl(  \frac{x_2}{{x_{12}}^2} +  \frac{x_2}{{x_{23}}^2} \biggr),} 
\\  \\
r_{12} &= &  \displaystyle{\frac{1}{ 2 x_{12}} +\frac{1}{10}\biggl(\frac{x_{12}}{x_{13} x_{23}} - \frac{x_{13}}{x_{12} x_{23}} - \frac{x_{23}}{x_{12} x_{13}}\biggr) } 
\displaystyle{-
\frac{1}{20} \biggl(   \frac{x_1}{{x_{12}}^2} + 2 \frac{x_2}{{x_{12}}^2} \biggr)}, 
\\  \\
r_{23}  &= &  \displaystyle{\frac{1}{ 2 x_{23}} +\frac{1}{10}\biggl(\frac{x_{23}}{x_{13} x_{12}} - \frac{x_{13}}{x_{12} x_{23}} - \frac{x_{12}}{x_{23} x_{13}}\biggr) -
\frac{1}{10}   \frac{x_2}{{x_{23}}^2}}
\\  \\
r_{13}  &= &  \displaystyle{\frac{1}{  2 x_{13}} +\frac{3}{20}\biggl(\frac{x_{13}}{x_{12} x_{23}} - \frac{x_{12}}{x_{13} x_{23}} - \frac{x_{23}}{x_{12} x_{13}}\biggr) -
\frac{1}{20}   \frac{x_1}{{x_{13}}^2}  }.  
\end{array}  \right\}
\end{equation}
}
 
 We consider the system of equations   
 \begin{equation}\label{eq30a} 
 r_1 = r_2, \,\,\,\,  r_2 =  r_{12}, \,\, \,\, r_{12} = r_{23}, \,\,\,\,  r_{23} = r_{13}. 
 \end{equation}
Then finding Einstein metrics of the form (\ref{metric1}) reduces  to finding the positive solutions of system (\ref{eq30a}),  and  we normalize  our equations by putting $x_{23}=1$. Then we have the system of equations: 
{\small \begin{equation}\label{eq31a}
\left. { \begin{array}{lll}
g_1 &=& 2 {x_{1}} {x_{12}}^2 {x_{2}}+3 {x_{1}}
   {x_{13}}^2 {x_{2}}-2 {x_{12}}^2 {x_{13}}^2
   {x_{2}}^2 -{x_{12}}^2  {x_{13}}^2 -2
   {x_{13}}^2 {x_{2}}^2 = 0,  \\ 
g_2 &=&  {x_{1}} {x_{13}}
   {x_{2}}-2 {x_{12}}^3 {x_{2}}+2 {x_{12}}^2
   {x_{13}} {x_{2}}^2+{x_{12}}^2 {x_{13}}+2
   {x_{12}} {x_{13}}^2 {x_{2}} \\
  &&  -10 {x_{12}}
   {x_{13}} {x_{2}}+2 {x_{12}} {x_{2}}+4
   {x_{13}} {x_{2}}^2 = 0, \\ 
g_3 &=& -{x_{1}} {x_{13}}+4
   {x_{12}}^3+2 {x_{12}}^2 {x_{13}} {x_{2}}-10
   {x_{12}}^2 {x_{13}}+10 {x_{12}} {x_{13}}-4
   {x_{12}}-2 {x_{13}} {x_{2}} = 0, \\
 g_4 &=& {x_{1}}
   {x_{12}}+{x_{12}}^2 {x_{13}}-2 {x_{12}}
   {x_{13}}^2 {x_{2}}+10 {x_{12}} {x_{13}}^2-10
   {x_{12}} {x_{13}}-5 {x_{13}}^3+5
   {x_{13}} = 0. 
\end{array} } \right\} 
\end{equation} }

   We consider a polynomial ring $R= {\mathbb Q}[z, x_1, x_2, x_{12}, x_{13}] $ and an ideal $I$ generated by 
$\{ g_1, \, g_2, $  $ g_3, g_4,  \,z \,x_1 \, x_2 \, x_{12} \, x_{13} -1\}  
$  to find non-zero solutions of equations (\ref{eq31a}). 
We take a lexicographic order $>$  with $ z > x_1 >  x_2 > x_{12} > x_{13}$ for a monomial ordering on $R$. Then, by the aid of computer, we see that a  Gr\"obner basis for the ideal $I$ contains the  polynomial  
$(x_{13} - 1) \, h_1(x_{13}),$
where $h_{1}(x_{13})$ is a polynomial of   $x_{13}$  given by 
{\small 
\begin{eqnarray*}  %& & 
h_{1}(x_{13}) &=& 688046498713728 {x_{13}}^{22}-5679627129033984 
   {x_{13}}^{21}+21187741726130976
   {x_{13}}^{20} \\  & & 
     -47332135234207584 {x_{13}}^{19}+72943727603815728 
   {x_{13}}^{18}-88042204949117760 
   {x_{13}}^{17} \\  & &    +90811237969386720 
   {x_{13}}^{16}-75973652107795440 
   {x_{13}}^{15}+40485498601824360 
   {x_{13}}^{14} \\  & & 
   -388980702921240 
   {x_{13}}^{13}-24758853711650442 
   {x_{13}}^{12}+32500688684143066 
   {x_{13}}^{11} \\  & & 
   -27877210026039119 
   {x_{13}}^{10}+14899625214395426 
   {x_{13}}^9-1186420879578712 
   {x_{13}}^8 \\  & &
   -6317807798571000 
   {x_{13}}^7 +7671384589125120 
   {x_{13}}^6-6094614793248000 
   {x_{13}}^5\\  & & 
   +3670257014726400 
   {x_{13}}^4-1592931826944000 
   {x_{13}}^3+457180443648000 
   {x_{13}}^2 \\  & & 
   -76918947840000 {x_{13}}+5733089280000. 
\end{eqnarray*} }
% Note that, if  the equation $h_{1}(x_{13})  =0$ has real solutions, then these are positive. 
 If $x_{13}\ne 1$ we solve the equation $ h_{1}(x_{13})=0$ numerically, and we obtain two positive solutions $x_{13}= a_{13}$ and $x_{13}= b_{13}$ which are given approximately as
 $ a_{13} \approx  1.13934, \, \,  b_{13} \approx 0.350124. $
 We also see that the Gr\"obner basis for the ideal $I$ contains the polynomials 
 $$x_{12} - w_{12}(x_{13}), \quad x_1 -w_1 (x_{13}), \quad x_2 -w_2 (x_{13}), $$
 where $ w_{12}(x_{13})$, $ w_{1}(x_{13})$ and $ w_{2}(x_{13})$ are polynomials of $x_{13}$ with rational coefficients. By substituting the values $ a_{13}$ and  $b_{13}$ for $x_{13}$  into $w_{12}(x_{13})$, $ w_{1}(x_{13})$ and $w_{2}(x_{13})$, we obtain two solutions of the system of equations $\{ g_1=0, g_2=0, g_3=0,  g_4=0  \}$ approximately given as
 $$(x_{13}, x_{12}, x_1,  x_2)  \approx (1.13934, \,0.620201, \, 0.831771, \, 0.149407 ), $$ $$ (x_{13}, x_{12},  x_1,  x_2)  \approx ( 0.350124, \, 1.03223, \, 0.455639, \, 0.121264). 
 $$
 If $x_{13}=1$ we see that   $x_{12} = x_1$, $x_{12} = x_2$  and $3 - 10 x_2 + 6 { x_2}^2 =0$. 
 Thus we obtain  two more solutions of the system of equations $\{ g_1=0, g_2=0, g_3=0,  g_4=0  \}$ as 
 $$(x_{13}, x_{12}, x_1, x_2)  = \left( 1,  ( 5+ \sqrt{7})/6, ( 5+ \sqrt{7})/6,  ( 5+ \sqrt{7})/6 \right),$$ $$(x_{13}, x_{12}, x_1, x_2 ) = \left( 1,  (  5 - \sqrt{7})/6, ( 5 - \sqrt{7})/6, ( 5 - \sqrt{7})/6 \right), 
 $$
 which are  known as the Jensen's Einstein metrics on Stiefel manifolds.
  
\medskip

 \underline{Case 2.}  $\Ad( \SO(4)\times\SO(2))$-invariant products of the form (\ref{metric2})

 From Lemma \ref{ricV4Rn} we see that  the components  of  the Ricci tensor ${r}$ for the invariant metric are given by 
{\small 
 \begin{equation}\label{eq26}
\left. {\small \begin{array}{l} 
 r_2  = 
\displaystyle{\frac{1}{10 x_2} +
\frac{1}{20} \biggl(  \frac{x_2}{{x_{12}}^2} +2 \frac{x_2}{{x_{23}}^2} \biggr)},
\\  \\
r_{12}   =  \displaystyle{\frac{1}{2 x_{12}} +\frac{1}{10}\biggl(\frac{x_{12}}{x_{13} x_{23}} - \frac{x_{13}}{x_{12} x_{23}} - \frac{x_{23}}{x_{12} x_{13}}\biggr) -
\frac{3}{20}  \frac{x_2}{{x_{12}}^2}},
\\  \\
r_{23}   =  \displaystyle{\frac{1}{2 x_{23}} +\frac{1}{20}\biggl(\frac{x_{23}}{x_{13} x_{12}} - \frac{x_{13}}{x_{12} x_{23}} - \frac{x_{12}}{x_{23} x_{13}}\biggr) -
\frac{3}{20} \frac{x_2}{{x_{23}}^2}},
\\  \\
r_{13}   =  \displaystyle{\frac{1}{ 2 x_{13}} +\frac{1}{5}\biggl(\frac{x_{13}}{x_{12} x_{23}} - \frac{x_{12}}{x_{13} x_{23}} - \frac{x_{23}}{x_{12} x_{13}}\biggr) }.
\end{array} } \right\} 
\end{equation} }
 We consider the system of equations   
 \begin{equation}\label{eq27a} 
 r_2 =  r_{12}, \,\,  r_{12} = r_{23}, \,\,  r_{23} = r_{13}. 
 \end{equation}
Then finding Einstein metrics of the form (\ref{metric2}) reduces  to finding positive solutions of system (\ref{eq27a}),  and  we normalize  our equations by putting $x_{23}=1$. Then we have the system of equations: 
{\small 
 \begin{equation}\label{eq27b}
\left. { \begin{array}{lll}  
f_1 &=& {x_{12}}^3 (-{x_{2}})+{x_{12}}^2 {x_{13}} {x_{2}}^2+{x_{12}}^2{x_{13}}+{x_{12}} {x_{13}}^2{x_{2}}
%\\ & & 
-5 {x_{12}} {x_{13}} {x_{2}} \\  && +{x_{12}} {x_{2}}+2 {x_{13}} {x_{2}}^2=0,  \\  
f_2 &=& 3 {x_{12}}^3+3 {x_{12}}^2 {x_{13}} {x_{2}}-10{x_{12}}^2 {x_{13}}-{x_{12}} {x_{13}}^2  +10 {x_{12}} {x_{13}}-3{x_{12}}-3 {x_{13}} {x_{2}}=0, \\  
f_3 &=& 3 {x_{12}}^2-3 {x_{12}} {x_{13}} {x_{2}}+10 {x_{12}} {x_{13}}-10{x_{12}}-5 {x_{13}}^2+5=0. 
\end{array} } \right\} 
\end{equation} }

 We consider a polynomial ring $R= {\mathbb Q}[z,  x_2, x_{12}, x_{13}] $ and an ideal $I$ generated by 
$\{ f_1, \, f_2, \, f_3, $  $  \,z  \, x_2 \, x_{12} \, x_{13} -1\}  
$  to find non-zero solutions of equations (\ref{eq27b}). 
We take a lexicographic order $>$  with $ z >   x_2 > x_{12} > x_{13}$ for a monomial ordering on $R$. Then, by the aid of computer, we see that a  Gr\"obner basis for the ideal $I$ contains the 
 polynomial  
$(x_{13} - 1) \, h_{2}(x_{13}),$
where $h_{2}(x_{13})$ is a polynomial of   $x_{13}$  given by 
{\small 
\begin{eqnarray*}   
h_{2}(x_{13}) &  = & 78808464 {x_{13}}^{10}-391536432
   {x_{13}}^9+848044372
   {x_{13}}^8-1028743244
   {x_{13}}^7 \\& &
   +802028465
   {x_{13}}^6-565003906
   {x_{13}}^5+511844730
   {x_{13}}^4-416144424
   {x_{13}}^3 \\ & & +210074472
   {x_{13}}^2-55497312
   {x_{13}}+5668704
 \end{eqnarray*} }
 %Note that, if  the equation $h_{2}(x_{13})  =0$ has real solutions, then these are positive. 
 If $x_{13}\ne 1$ we solve the equation $ h_{2}(x_{13})=0$ numerically, and we obtain two positive solutions $x_{13}= \alpha_{13}$ and $x_{13}= \beta_{13}$ which are given approximately as
 $ \alpha_{13} \approx  0.253386, \quad \beta_{13} \approx 1.16137. $
 We also see that the Gr\"obner basis for the ideal $I$ contains polynomials 
 $$x_{12} - w_{12}(x_{13}), \quad x_2 -w_2 (x_{13}), $$
 where $ w_{12}(x_{13})$ and $ w_{2}(x_{13})$ are polynomials of $x_{13}$ with rational coefficients. By substituting the values $  \alpha_{13}$ and  $\beta_{13}$ for $x_{13}$  into $w_{12}(x_{13})$ and $w_{2}(x_{13})$, we obtain two solutions of the system of equations $\{ f_1=0, f_2=0, f_3=0  \}$ approximately as
 $$(x_{13}, x_{12}, x_2)  \approx (0.253386, \,1.01652, \, 0.245146), \,\, (x_{13}, x_{12}, x_2)  \approx ( 1.16137, \, 0.669071, \, 0.291175). 
 $$
 If $x_{13}=1$ we see that  $x_{12} = x_2$  and $3 - 10 x_2 + 6 { x_2}^2 =0$. 
 %Thus we  obtain two positive  solutions  $\displaystyle  x_2= ( 4+ \sqrt{6})/5, \quad x_2= ( 4- \sqrt{6})/5.$
 Thus we obtain  two more solutions of the system of equations $\{ f_1=0, f_2=0, f_3=0  \}$ as 
 $$(x_{13}, x_{12}, x_2)  = \left( 1,  ( 5+ \sqrt{7})/6, ( 5+ \sqrt{7})/6 \right), \,  (x_{13}, x_{12}, x_2 ) = \left( 1,  (  5 - \sqrt{7})/6, (  5 - \sqrt{7})/6 \right), 
 $$
which are the same metrics as in Case 1. 

\end{proof}


\begin{thebibliography}{50}
  
  
\bibitem [AbKo]{AK}M. T. K. Abbassi - O. Kowalski: {\it Naturality of
homogeneous matrics on Stiefel manifolds}, Differential Geometry and its
Appl. 28 (2010) 131-139.
  
  
\bibitem [AlDoFe]{ADF} 
 D. V. Alekseevsky, I. Dotti and C. Ferraris: 
{\it Homogeneous Ricci positive $5$-manifolds}, 
Pacific. J. Math. 175 (1) (1996) 1--12. 



  \bibitem [Ar1]{A1} 
A. Arvanitoyeorgos:
{\it Homogeneous Einstein metrics on Stiefel manifolds},
Comment. Math. Univ. Carolinae 37 (3) (1996) 627--634. 

 
 
  \bibitem [ArDzNi1]{ADN1} 
A. Arvanitoyeorgos, V.V. Dzhepko and Yu. G. Nikonorov:
{\it Invariant Einstein metrics on some homogeneous spaces of classical Lie groups}, 
 Canad. J. Math. 61 (6) (2009) 1201--1213.
 
\bibitem [ArDzNi2]{ADN2} 
A. Arvanitoyeorgos, V.V. Dzhepko and Yu. G. Nikonorov:
{\it Invariant Einstein metrics on certain Stiefel manifolds},
Differential Geometry and its Applications, Proc. Conf., in Honor of Leonard Euler, Olomouc, August 2007, World Sci. 2008, 35--44.

\bibitem [ArDzNi3]{ADN3} 
A. Arvanitoyeorgos, V.V. Dzhepko and Yu. G. Nikonorov:
{\it Invariant Einstein metrics on quaternionic Stiefel manifolds},
Bull. Greek Math. Soc. 53 (2007) 1--14.

 
 \bibitem[BaHs]{BH}
 A. Back and W.Y. Hsiang:
 {\it Equivariant geometry and Kervaire spheres}, Trans. Amer. Math. Soc. 304 (1) (1987) 207--227.
  
\bibitem[Be]{Be} 
A. L. Besse: 
 {\it Einstein Manifolds}, 
  Springer-Verlag, Berlin, 1986. 
  
  \bibitem[B\"o]{Bom}  C. B\"ohm: {\it Homogeneous Einstein metrics and simplicial complexes}, J. Differential Geom. 67 (1) (2004) 79--165. 
  
 \bibitem[B\"oWaZi]{BWZ}
 C. B\"ohm, M. Wang and W. Ziller:
  {\it A variational approach for compact homogeneous Einstein manifolds}, Geom. Func. Anal. 14 (4) (2004) 681--733.

 \bibitem[DAZi]{DZ} J. E. D' Atri and W. Ziller:  {\it Naturally reductive metrics and Einstein metrics on compact Lie groups},  Memoirs Amer. Math. Soc. 19 (215) (1979). 
  
\bibitem[Je]{J2}
G. Jensen:
{\it Einstein metrics on principal fiber bundles},
J. Differential Geom. 8 (1973) 599--614.

\bibitem[Ke]{Ke}
M. Kerr:
{\it New examples of homogeneous Einstein metrics}, Michigan J. Math. 45 (1) (1998) 115--134.

\bibitem[Ko]{K}
S. Kobayashi:
{\it Topology of positively pinched K\"ahler manifolds}, T\^ohoku Math. J. 15 (1963) 121--139.
    
    \bibitem[PaSa]{PS}
  J-S. Park and Y. Sakane:
  {\it Invariant Einstein metrics on certain homogeneous spaces},
  Tokyo J. Math.  20   (1) (1997) 51--61.
  
  \bibitem[Sa]{S}
  A. Sagle:
  {\it Some homogeneous Einstein manifolds}, Nagoya Math. J. 39 (1970) 81--106.
  
\bibitem[Wa1]{W1}
M. Wang: 
{\it Einstein metrics from symmetry and bundle constructions}, In: Surveys in Differential Geometry: Essays on Einstein Manifolds. Surv. Differ. Geom. VI, Int. Press, Boston, MA 1999.

\bibitem[Wa2]{W2}
M. Wang: 
{\it Einstein metrics from symmetry and bundle constructions: A sequel}, In: 
 Differential Geometry: Under the Influence of S.-S. Chern, Advanced Lectures in Math., Vol. 22,
Higher Education Press/International Press, 2012, 253--309 .  
  
\bibitem[WaZi]{WZ}
M. Wang and W. Ziller: 
{\it Existence and non-excistence of homogeneous Einstein metrics}, 
Invent. Math. 84 (1986)  177--194.
  \end{thebibliography}
 \end{document}